\documentclass[12pt]{amsart}

\usepackage{amsmath}
\usepackage{amsfonts}
\usepackage{amssymb}
\usepackage{amsthm} 
\usepackage[english]{babel}
\usepackage[latin1]{inputenc}
\usepackage[T1]{fontenc}
\usepackage{graphicx}
\usepackage{enumitem}

\usepackage{geometry}
\usepackage{hyperref}

\DeclareMathOperator{\Ric}{Ric}

\DeclareMathOperator{\spt}{spt}
\DeclareMathOperator{\Id}{Id}

\DeclareMathOperator{\vol}{vol}

\newtheorem{theo}{Theorem}[]
\newtheorem{prop}[theo]{Proposition}
\newtheorem{lemme}[theo]{Lemma}
\newtheorem{definition}[theo]{Definition}
\newtheorem{coro}[theo]{Corollary}

\newtheorem{remarque}[theo]{Remark}
\begin{document}

\title[Appearance of stable 
minimal spheres along the Ricci flow]{Appearance of stable 
minimal spheres along the Ricci flow in positive scalar curvature} 
\author{Antoine Song}

\begin{abstract}
We construct examples of spherical space forms $(S^3/\Gamma,g)$ with positive scalar curvature and containing no stable embedded minimal surfaces, such that the following happens along the Ricci flow starting at $(S^3/\Gamma,g)$: a stable embedded minimal two-sphere appears and a non-trivial singularity occurs. We also give in dimension $3$ a general contruction of Type I neckpinching and clarify the relationship between stable spheres and non-trivial Type I singularities of the Ricci flow. Some symmetry assumptions prevent the appearance of stable spheres, and this has consequences on the types of singularities which can occur for metrics with these symmetries.

\end{abstract}

\maketitle

\section*{} 

For quotients of the spheres of dimension $2$ and $3$, endowed with an arbitrary metric, the Ricci flow eventually makes the metric converge to a round metric. In dimension $2$, it was proved by Hamilton \cite{Hamiltonsurface} and B. Chow \cite{Chowsphere} that starting at any metric, the latter evolves smoothly under the Ricci flow until a trivial singularity, where the whole surface disappears at a point and after rescaling becomes asymptotically round. The situation is way more complicated in dimension $3$ because non-trivial singularities can occur. In a serie of papers \cite{Perelman1} \cite{Perelman2} \cite{Perelman3}, Perelman was able to analyse and control the singularities by a surgery process initially proposed by Hamilton, which enables to continue the flow. One simple consequence of this breakthrough is that for quotients of the $3$-sphere, after a finite number of surgeries, the manifold disappears in finite time, and also becomes asymptotically round. From a related point of view, if the initial metric is already known to be round enough, then it becomes even more so during the flow: this is Hamilton's theorem \cite{Hamilton} which states that if a closed $3$-manifold has positive Ricci curvature then this property is preserved and after rescaling the metric converges smoothly to a round metric. Besides, it is well known that the positivity condition $\Ric>0$ prevents the existence of two-sided closed stable minimal surfaces. Hence it is natural to ask if the absence of such stable minimal surfaces will also be preserved along the flow. 

Furthermore the study of stable minimal surfaces in the context of the Ricci flow can be motivated by the attempt to better understand singularity formation. For instance the heuristic picture for the Ricci flow on a $3$-sphere is that a non-degenerate singularity which is non-trivial should be a neckpinching, thus there should be small stable minimal spheres just before the singularity time. To our knowledge, the only rigourously proved examples of initial metrics on the $3$-sphere eventually producing a non-trivial singularity \cite{AngenentKnopf} \cite{AngenentKnopf2}, contain a stable minimal sphere. Thus one might hope to avoid non-trivial singularities if the initial metric does not contain stable minimal surfaces.

It will be enough for us to focus on the case where the scalar curvature is positive. This condition $R>0$ is considerably weaker than $\Ric >0$ but nevertheless conveys an idea of roundness and is preserved along the flow. Notice that if $R>0$, any two-sided oriented closed stable minimal surface is a $2$-sphere. Let us reformulate the two previous questions:
$$$$

\textbf{Q1:} Suppose that $(M,g)$ is a closed oriented $3$-manifold with positive scalar curvature and containing no stable minimal spheres. Can a stable minimal sphere appear along the Ricci flow starting at $(M,g)$?

\textbf{Q2: } Let $(M,g)$ be as in $\textbf{Q1}$. Can a non-trivial singularity occur along the Ricci flow starting at $(M,g)$?
$$$$
It turns out that the answer to both questions is yes and the examples are the object of our main theorem (see Theorems \ref{appearance2} and \ref{appearance3}): 

\begin{theo}
There exists a metric $g$ on $S^3$ with positive scalar curvature such that
\begin{enumerate}
\item  $(S^3,g)$ contains no stable minimal $2$-spheres, 
\item a stable minimal $2$-sphere appears along the Ricci flow starting at $(S^3,g)$,
\item a non-trivial singularity occurs in finite time.  
\end{enumerate}

\end{theo}
This shows that, even when $R>0$, the absence of stable spheres at the beginning cannot prevent non-trivial singularities. Actually the appearance of stable geodesics is also true for some $2$-spheres (see Theorem \ref{appearance1}), but of course we cannot impose a curvature positivity condition in that case since it would be preserved by the flow and this would prevent the existence of stable closed geodesics.

Let us point out how Question \textbf{Q1} is related to issues concerning the min-max theory for minimal surfaces. Using the Ricci flow, Marques and Neves proved a $3$-dimensional Toponogov type theorem (see Theorem 1.3 in \cite{MaNe}). In particular they suppose that $\Ric >0$ and by controlling the evolution of a min-max width along the Ricci flow, they are able to produce a small area minimal surface for the initial metric. While in \cite{Antoine} we proved using a different method that this result remains true in the general case $R>0$ without further assumption on the Ricci curvature, it would be desirable to understand to what extent their combination of Ricci flow and min-max theory can be realized when $\Ric$ is not necessarily positive. The reason of this assumption $\Ric>0$ in \cite{MaNe} is twofold. First they are making use of Hamilton's theorem so that they do not have to deal with surgeries. The second and most serious reason for this assumption is the following: as recalled previously, it excludes the existence of two-sided stable minimal surfaces so in particular it enables the construction of optimal sweepouts from a given unstable two-sided minimal surface. This observation has been repeatedly used in recent applications of the min-max theory (\cite{MaNe},\cite{MaNeinfinity}, \cite{Zhou}, \cite{Zhou2}, \cite{KeMaNe}). The examples that we construct to answer $\textbf{Q1}$ suggest that combining min-max theory with the Ricci flow when $R>0$ may not be as natural as in the more restrictive case $\Ric>0$.

After answering Questions \textbf{Q1} and \textbf{Q2}, we clarify the link between small stable spheres and Type I singularites in dimension $3$, without curvature assumptions (see Theorem \ref{spherestypeI}): 

\begin{theo}
If a non-trivial Type I singularity occurs at time $T$, then there are stable immersed minimal spheres with embedded image near time $T$ whose area decreases linearly to zero, and a local converse holds true. 
\end{theo}

Moreover, we construct in Proposition \ref{singularity neck} general examples of Type I neckpinching by joining any two closed $3$-manifolds with a sufficiently thin neck, which generalizes in dimension $3$ the rotationally symmetric metrics on $S^{n+1}$ constructed by Angenent and Knopf \cite{AngenentKnopf}.   

Finally, one can ensure that no stable spheres appear if the initial metric is symmetric enough (see Theorem \ref{symmetry}):
\begin{theo} Let $M$ be a closed connected oriented $3$-manifold and a $d$-dimensional Lie group of isometries acting on $M$ such that $d>1$, or $d=1$ and the action is free in that case. Then along the Ricci flow, no stable immersed minimal spheres with embedded image can appear if there were none at the beginning.
\end{theo}

In fact, there is a non-free $S^1$ action on the examples constructed previously to answer $\textbf{Q1}$. Hence these examples essentially have a maximal amount of symmetry among $3$-manifolds such that new stable spheres appear along the Ricci flow. Since we have seen that stable spheres are linked to non-trivial Type I singularities, we get as a corollary (see Corollary \ref{cocoo}):

\begin{coro}
When $M$ as in the previous theorem is not rotationally symmetric and if a singularity occurs along the Ricci flow, then it is a Type I trivial singularity.
\end{coro}

The paper is organized as follows. In Section 1, some preliminaries on the Ricci flow and min-max theory are presented. Section 2 constitutes the main part of the article, certain "thin hooks" are constructed, and this enables to give explicit examples answering at the same time \textbf{Q1} and \textbf{Q2}. In Section 3, after showing a general procedure to get Type I neckpinching, we prove a relation between stable spheres and non-trivial Type I singularities. Finally, we propose in Section 4 symmetry assumptions preventing the appearance of new stable spheres and a corollary concerning singularities which can possibly occur is derived.

\subsection*{Acknowledgement} 
I am grateful to my advisor Fernando Cod{\'a} Marques for his support and his helpful remarks. I also thank Ian Agol for sharing his former thoughts about these questions related to the study of singularities, and Otis Chodosh, John Lott, Richard Bamler for their interest.

\section{Preliminaries}

\subsection{Ricci flow, singularities, canonical neighborhoods and geometric limits} \label{RF}

Let $(M,g)$ be a closed oriented Riemannian $3$-manifold. A standard Ricci flow starting at $(M,g)$, defined on $[0,T)$, is a smooth solution of 
$$
\left\{ \begin{array}{rcl}
&\frac{\partial}{\partial t} g(t)  = -2\Ric_{g(t)}, \quad t\in[0,T)\\
& g(0) = g,
\end{array}\right.$$
The flow is said to develop a singularity at time $T$ if the norm of the curvature tensor goes to infinity as $t\to T$. The singularity is said to be 
\begin{itemize}
\item \textit{trivial} when 
$$\{x\in M ; \lim_{t\to T}|Rm(x,t)|=\infty\} = M,$$
\item a \textit{Type I} singularity when there is a constant $\bar{C}$ such that for all $t\in [0,T)$,
$$\sup_M|Rm_{g(t)}|\leq \frac{\bar{C}}{T-t},$$
\item a \textit{Type II} singularity when 
$$\limsup_{t\uparrow T} \sup_M |Rm_{g(t)}| (T-t)= \infty.$$
\end{itemize}

As Perelman showed, the regions where the scalar curvature is large are modelled by the so-called canonical neighborhoods. To explain their properties, we will use \cite{MorganTian} (for other references, see also \cite{K&L}, \cite{CaoZhu}). Recall the definition of $(C,\epsilon)$-canonical neighborhoods. Fix $C$ and $\epsilon$ two positive constants. An open neighborhood $U$ of $x\in (M,g(t))$ is a strong $(C,\epsilon)$-canonical neighborhood if one of the following holds (see \cite[Section 8 in Chapter 9 and Definition 14.18]{MorganTian}):
\begin{enumerate}
\item $U$ is a strong $\epsilon$-neck in $(M,g)$ centered at $x$,
\item $U$ is a $(C,\epsilon)$-cap in $(M,g)$ whose core contains $x$,
\item $U$ is a $C$-component of $(M,g)$ satisfying Condition (8) of \cite[Definition 9.72]{MorganTian},
\item $U$ is an $\epsilon$-round component of $(M,g)$.
\end{enumerate}
A strong $\epsilon$-neck centered at $x\in (M,g(t))$ is a submanifold $N\subset M$ and a diffeomorphism $\bar{\psi}_N : S^2\times (-1/\epsilon,1/\epsilon)\to N$ such that $t-R(x,t)\geq 0$ and the evolving metric $R(x,t)\bar{\psi}^*(g(t+s/R(x,t)))$, $-1< s \leq 0$, is $\epsilon$-close in the $C^{[1/\epsilon]}$-topology to the evolving cylindrical metric $ds^2+d\theta)^2$, $-1< s \leq 0$, where $d\theta^2$ denotes the round metric of scalar curvature $1/(1-s)$ on $S^2$. A $(C,\epsilon)$-cap is a noncompact submanifold $\mathcal{C} \subset M$ diffeomorphic to a $3$-ball or to $\mathbb{R}P^3$ minus a ball, with a neck $N\subset \mathcal{C}$ such that $\bar{Y}=\mathcal{C}\backslash N$ is a compact submanifold called core. The boundary $\partial \bar{Y}$ of the so-called core (the interior of $\mathcal{C}\backslash N$) is required to be the central sphere of a strong $\epsilon$-neck in $\mathcal{C}$. After rescaling the metric to have $R(x)=1$ at some point $x$ in the cap, the diameter, volume and scalar curvature ratios at any two points are bounded by $C$. A $C$-component is a compact manifold diffeomorphic to $S^3$ or $\mathbb{R}P^3$, of positive sectional curvature and of bounded geometry controlled by $C$ after rescaling (we added Condition (8) of \cite[Definition 9.72]{MorganTian} since the next theorem is actually true with this definition). An $\epsilon$-round component is a compact connected manifold such that, after rescaling to make $R(x)=1$ at some point $x\in M$, the metric is close in the $C^{[1/\epsilon]}$-topology to a round metric. The definition of strong $(C,\epsilon)$-canonical neighborhoods is hence scale invariant. 

We say that a (standard) Ricci flow $(M,g(t))_{t\in[a,b)}$ satisfies the $(C,\epsilon)$-canonical neighborhood assumption with parameter $r$ if every point $(x,t)\in M\times [a,b)$ with $R_{g(t)}(x) \geq r^{-2}$ has a $(C,\epsilon)$-canonical neighborhood. When $\epsilon$ and $1/C$ are small enough, one has the following canonical neighborhood theorem (see for instance \cite[Chapter 9, Chapter 17 and Theorem 15.9]{MorganTian}):

\begin{theo}  \label{canoneighbo}
Let $T>0$. Then there exists an $r_0>0$ depending only on $T$ such that the following holds. 
Suppose that $(M,g)$ is a closed oriented $3$-manifold endowed with a normalized metric, i.e. for all $x\in M$:
$$\max_{M}|Rm(x,0)|_{g(0)}\leq 1 \text{ and } \vol_{g(0)} B(x,0,1) \geq \omega/2,$$
where $\omega$ is the volume of the unit ball in $\mathbb{R}^3$.
Assume that the Ricci flow $(M_t,g(t))_{t\in[0,t_1)}$ is well-defined until a time $t_1\leq T$. Then $(M_t,g(t))_{t\in[0,t_1)}$ satisfies the strong $(C,\epsilon)$-canonical neighborhood assumption with parameter $r_0$. 
\end{theo}

The relationship between the Type I/II classification and the canonical neighborhoods was given in \cite{YDing}: a singularity at time $T$ is of Type II if and only if there is a sequence $(x_k,t_k)$ with $x_k\in M$, $t_k\to T$ such that the scalar curvature at $(x_k,t_k)$ goes to infinity and $(x_k,t_k)$ is contained in a $(C,\epsilon)$-cap diffeomorphic to a $3$-ball (it corresponds to $iv)$ in \cite[Proposition 1.4]{YDing}). This geometric characterization of Type II singularities will be useful. We note that the other kind of $(C,\epsilon)$-caps, those diffeomorphic to $\mathbb{R}P^3$ minus a point, have a double cover which is a strong $\epsilon$-neck.

The scalar curvature evolves according to
$$\frac{\partial R}{\partial t} = \Delta R + 2|\Ric|^2.$$
Thus when $x$ is in a strong $\epsilon$-neck at time $t$ or a $(C,\epsilon)$-cap diffeomorphic to $\mathbb{R}P^3$ minus a point, there is a positive constant $C_1$ so that 
\begin{equation} \label{neck derivative}
\frac{1}{C_1} R(x,t)^2 \leq \frac{\partial R(x,t)}{\partial t} \leq C_1 R(x,t)^2.
\end{equation}
and there is a positive constant $C_2$ such that whenever $x$ is in a strong canonical neighborhood,
\begin{equation} \label{derivative estimate}
|\frac{\partial R(x,t)}{\partial t}| \leq C_2 R(x,t)^2.
\end{equation}

Along the Ricci flow, as the scalar curvature gets large it controls the whole curvature tensor. Let $(M,g(t))_{t\in[0,T)}$ be a Ricci flow such that for all $x\in M$ the smallest eigenvalue of $Rm(x,0)$, denoted by $\nu(x,0)$, is at least $-1$. Set $X(x,t) = \max(-\nu(x,t),0)$. Then Ivey \cite{Iveypinching} and Hamilton \cite{Hamiltonpinching} showed the following "pinching towards positive" property:
\begin{theo}
We have the following properties.
\begin{enumerate}
\item$ R(x,t)\geq \frac{-6}{4t+1}$ and
\item for all $(x,t)$ for which $0<X(x,t)$,
$$R(x,t)\geq 2X(x,t)(\log X(x,t)+ \log(1+t) -3).$$
\end{enumerate}
\end{theo}

The notion of geometric convergence \cite[Chapter 5]{MorganTian} describes the convergence of based Ricci flows, and can be extended to any time interval (i.e. to intervals not of the form $(-T,0]$). We will need the following convergence property of the Ricci flow. 

\begin{lemme} \label{geomconv}
Let $T>0$, let $(M_k,g_k)$ be a sequence of closed normalized $3$-manifolds. Suppose that for any sequence of points $\mathfrak{s} =\{x_k\}_k$ where $x_k \in M_k$, the following holds. Subsequentely the sequence of based manifolds $(M_k,g_k,x_k)$ converges geometrically to a complete based manifold $(M^{\mathfrak{s}}_\infty, g^{\mathfrak{s}}_\infty,x_\infty)$ such that 
\begin{enumerate}
\item the Ricci flow $(M^{\mathfrak{s}}_\infty, g^{\mathfrak{s}}_\infty(t))$ with initial metric $g^{\mathfrak{s}}_\infty(0)=g^{\mathfrak{s}}_\infty$ exists, is unique, and defined for $0\leq t <T$,
\item for all $t<T$ there is a constant $C_0=C_0(t)$ independent of the sequence $\mathfrak{s}$ so that the norm of the curvature tensor of $(M^\mathfrak{s}_\infty, g_\infty^{\mathfrak{s}}(t'))$ is bounded by $C_0$ for all $t'\leq t$.
\end{enumerate}

Then for any sequence $\mathfrak{s} =\{x_k\}$ with $x_k\in M_k$, the sequence of based Ricci flows $(M_k,g_k(t),x_k)$ starting at $g_k(0)=g_k$ subsequently converges geometrically to $(M^{\mathfrak{s}}_\infty, g^{\mathfrak{s}}_\infty(t),x_\infty)$ on $[0,T)$.
\end{lemme}

\begin{proof}
Define
$$\tau = \sup\{t\in [0,T]; \exists C(t)>0, \forall t'\in [0,t],  \limsup_{k} \max_{M_k}|Rm(.,t')| \leq C(t)\}.$$
By (7.4a) and (7.4b) in \cite{ChowKnopf} and the argument in \cite[Lemma 6.1]{ChowLuNi}, we check that $\tau$ is positive and that for any integer $m$, there is a positive time $t_m$ for which $|\nabla^j Rm|$ ($0\leq j \leq m$) are bounded on $[0,t_m]$ uniformly in $k$. It follows from Shi's derivative estimates (see \cite[Chapter 6]{ChowLuNi} and \cite[Chapter 5]{MorganTian} for instance) that for all $\mathfrak{s} =\{x_k\}$ with $x_k\in M_k$, $(M_k,g_k(t),x_k)$ subsequently converges geometrically to $(M^{\mathfrak{s}}_\infty, g^{\mathfrak{s}}_\infty(t),x_\infty)$ on $[0,\tau)$ because of the first item in the assumptions. Hence it remains to show $\tau = T$. Suppose by contradiction that $\tau <T$ then, by Theorem \ref{canoneighbo} and (\ref{derivative estimate}), for all $C'>0$ there is a $\delta>0$ such that there are subsequences $M_{k(l)}$ and $x_l\in M_{k(l)}$ with the following property: the curvature at $(x_l,\tau-\delta)$ in $(M_{k(l)}, g_{k(l)}(\tau-\delta))$ has norm larger than $C'$. But by the geometric convergence on $[0,\tau)$ that was just explained and the second item in the assumptions, it is absurd when $C'>C_0(\tau) $. Thus $\tau=T$ and the lemma is proved.

\end{proof}

\subsection{Some min-max theory}

In this subsection, we present a variation of the min-max theorem in the continuous setting as described by De Lellis and Tasnady in \cite{Tasnady}.

Let $(M^{n+1},g)$ be a closed Riemannian manifold. In what follows, the topological boundary of a subset of $M$ will be denoted by $\partial$. Consider two open subsets $X$ and $N$ of $M$ possibly with a smooth boundaries, such that $(X\cup\partial X) \subset N$. The notation for the $m$-dimensional Hausdorff measure will be $\mathcal{H}^m$. Take $a<b$, $k\in \mathbb{N}$. 

\begin{definition} \label{definition}
A family of $\mathcal{H}^n$-measurable closed subsets $\{\Gamma_t\}_{tÊ\in [a,b]^k}$ in $N$ with finite $\mathcal{H}^n$-measure is called a generalized smooth family if
\begin{itemize}
\item for each $t $ there is a finite subset $P_t \subset N$ such that $\Gamma_t\cap N$ is a smooth hypersurface in $N\backslash P_t$,
\item $t \mapsto \mathcal{H}^n(\Gamma_t)$ is continuous and $t \mapsto \Gamma_t$ is continous in the Hausdorff topology,
\item $\Gamma_t \to \Gamma_{t_0}$ smoothly in any compact $U\subset\subset N\backslash P_{t_0}$ as $t\to t_0$. 
\end{itemize}
A generalized smooth family $\{\Sigma_t\}_{t\in [a,b]}$ is called a continuous sweepout in $N$ associated to $X$ if there exists a family of open subsets $\{\Omega_t\}_{t\in [a,b]}$ of $N$ such that
\begin{enumerate} [label=(\roman*)]
\item $(\Omega_t \cup \partial \Omega_t)\subset N$ for all $t\in[a,b]$,
\item $(\Sigma_t\backslash \partial\Omega_t) \subset P_t$ for any $t\in[a,b]$,
\item $\mathcal{H}^{n+1}(\Omega_t\backslash \Omega_s) + \mathcal{H}^{n+1}(\Omega_s\backslash \Omega_t) \to 0$, as $s\to t\in [a,b]$,
\item $\Omega_a =X$, and $\Omega_b= \varnothing$.
\end{enumerate}

\end{definition}

Still following \cite{Tasnady}, we define a notion of homotopy equivalence:

\begin{definition} \label{homotopic}
Two continuous sweepouts associated with $X$, $\{\Sigma_t^1\}_{t\in [a,b]}$ and $\{\Sigma_t^2\}_{t\in [a,b]}$, are homotopic if: 
\begin{itemize}
\item there is a generalized smooth family $\{\Gamma_{(s,t)}\}_{(s,t)\in [a,b]^2}$, such that $\Gamma_{(a,t)} =\Sigma_t^1$ and $\Gamma_{(b,t)} =\Sigma_t^2$ for all $t\in[a,b]$, 
\item $\Gamma_{(s,t)} \subset N$ for $t\in [a,b]$ and there exists a small $\alpha>0$ such that $\Gamma_{(s,t)}=\Gamma_{(a,t)}$ for $(s,t)\in [a,b]\times [a,a+\alpha]$. 
\end{itemize}

A family $\Lambda$ of continuous sweepouts in $N$ associated to $X$ is said to be homotopically closed if it contains the homotopy class of each of its element.

\end{definition}

If $\Lambda$ is a homotopically closed family of continuous sweepouts in $N$ associated with $X$, the width of $\Lambda$ in $N$ is defined as the min-max quantity
$$W(N,\partial N, \Lambda) = \inf_{\{\Sigma_t\}\in \Lambda} \max_t \mathcal{H}^n(\Sigma_t).$$
A sequence $\{\{\Sigma_t^k\}_{t\in [a,b]}\}_{k\in \mathbb{N}} \subset \Lambda$ is called a minimizing sequence if 
$$\max_t \mathcal{H}^n(\Sigma_t^k) \to W(N,\partial N,\Lambda)  \text{  as } k\to \infty.$$
A sequence of slices $\{\Sigma_{t_k}^k\}_{k \in \mathbb{N}}$ is called a min-max sequence if $$\mathcal{H}^n(\Sigma_{t_k}^k) \to W(N,\partial N, \Lambda)  \text{  as } k\to \infty.$$

The following theorem is a slight extension of \cite[Theorem 2.7]{Zhou}. It roughly says that if the beginning of sweepouts belonging to a homotopically closed family $\Lambda$ has $n$-volume less than the width of $\Lambda$, then the min-max theorem still holds as long as all the sweepouts are contained in an open set with mean convex boundary. Note that "mean convex" can be generalized to "piecewise smooth mean convex" (see \cite{Antoine}).
 
\begin{theo}  \label{two boundaries}

Let $(M,g)$ be a closed $(n+1)$-manifold with $2\leq n \leq 6$, and $N$, $X$ open subsets of $M$. Suppose that $\partial X\neq \varnothing$ and that $(X\cup \partial X) \subset N$. When $\partial N \neq \varnothing$, assume that $\partial N$ is mean convex. Then for any homotopically closed family $\Lambda$ of sweepouts in $N$ associated with $X$ such that 
$$W(N,\partial N,\Lambda)>\mathcal{H}^{n}(\partial X),$$ 
there exists a min-max sequence $\{\Sigma_{t_n}^n\}$ of $\Lambda$ converging in the varifold sense to an embedded minimal hypersurface $\Sigma$ (possibly disconnected), contained in $N$. Moreover the $n$-volume of $\Sigma$, if counted with multiplicities, is equal to $W(N,\partial N,\Lambda)$. 

\end{theo}

\begin{proof}

We essentially reproduce the proof of \cite[Theorem 2.7]{Zhou}. Recall that the latter is an application to higher dimensions of an idea in \cite{MaNe}, where the authors construct a vector field $\mathbf{V}$ in $N$ whose support is contained in a small neighborhood of $\partial N$ so that the corresponding flow is area decreasing. Thanks to this flow, they show that Proposition 4.1 in \cite{C&DL} still holds. What we modify here is that, in the proof of this proposition, we restrict ourselves to the set $\mathfrak{X}$ of varifolds whose mass is bounded above by $4W(N,\partial N,\Lambda)$ and also bounded below by $\mathcal{H}^{n}(\partial X) + \epsilon$, where $0< \epsilon <W(N,\partial N,\Lambda)-\mathcal{H}^{n}(\partial X)$. More precisely, let $\mathcal{V}_\infty$ be the set of stationary varifolds contained in $\mathfrak{X}$. By chosing a sufficiently fine locally finite covering of $\mathfrak{X}\backslash \mathcal{V}_\infty$, we construct for each varifold $V$ of mass less than $4W(N,\partial N,\Lambda)$ an ambient isotopy $\{\Psi_V(s,.)\}_{s\in[0,1]}$ satisfying the properties listed in Step 3 of the proof of \cite[Proposition 4.1]{C&DL} if $V\in\mathfrak{X}$ but such that $\Psi_V(s,.)=\Id$ for all $s\in[0,1]$ if the mass of $V$ is less than $\mathcal{H}^{n}(\partial X) + \epsilon/2$. Finally, by modifying $\{\Psi_V(s,.)\}_{s\in[0,1]}$ with the vector field $\mathbf{V}$ if necessary, we can deform a minimizing sequence $\{\{\Sigma^k_t\}_{t\in[0,1]}\}$ into an other minimizing sequence $\{\{\tilde{\Sigma}^k_t\}_{t\in[0,1]}\}$ such that all $\tilde{\Sigma}^k_t$ with area larger than $\mathcal{H}^{n}(\partial X) + \epsilon$ lie at bounded distance from $\partial N$. Then the end of the proof remains unchanged compared to \cite{Zhou}.
\end{proof}

\begin{remarque} \label{dimension 1}
If $n=1$, the following elementary version of Theorem \ref{two boundaries} will be useful. Suppose that $N$ and $X$ are diffeomorphic to the unit disk $D$ in $\mathbb{R}^2$, define $\{c_t\}_{ t\in[0,1]}$ as the smooth sweepout of $X$ obtained by the foliation $\{x\in\mathbb{R}^2 ; ||x||_{eucl}=t \}_{t\in[0,1]}$ of $D$, where $||.||_{eucl}$ is the Euclidean norm in $\mathbb{R}^2$. Let $\mathcal{C}$ be the space of smooth curves endowed with the $C^\infty$ topology. Let $\Lambda$ be the homotopically closed family of sweepouts $\{\tilde{c}_t\}_{ t\in[0,1]} \subset \mathcal{C}$ continuously isotope to $\{c_t\}_{ t\in[0,1]}$ in $N$ and such that $\tilde{c}_0=c_0$. Define $W(N,\partial N, \Lambda)$ as for the higher dimensional case. If $\partial N$ is convex and 
$$W(N,\partial N,\Lambda)>\mathcal{H}^{1}(\partial X),$$ 
then there is a simple closed geodesic in $N$ of length $W(N,\partial N,\Lambda)$. This can be proved using the mean curvature flow $\{\Phi(s,.)\}$ where $s$ is the time parameter. Define $\theta: \mathbb{R}\times\mathcal{C}\to \mathbb{R}$ such that 
\begin{align*}
\theta(s,c)=  \sup\{& s'\in[0,s] ; \Phi(s',c) \text{ has length} \\ & \text{ at least } \frac{1}{2}(W(N,\partial N,\Lambda)+\mathcal{H}^{1}(\partial X))\},
\end{align*}
where we use the convention $\sup \varnothing =0$.
Then given a minimizing sequence of sweepouts $\{\{c^n_t\}_{t\in[0,1]}\}$, we consider the new tightened sequence $\{\{\Phi({\theta(s,c^n_t)},c^n_t)\}_{t\in[0,1]}\}$ for each $s\geq0$. By the maximum principle, the new sweepouts are entirely contained in $N$. Letting $s\to \infty$, any min-max sequence converges subsequently to a simple closed geodesic inside $N$ (see \cite{Grayson}).

\end{remarque}

\section{Appearance of stable spheres and non-trivial singularities} \label{appear1}

\subsection{Construction of thin hooks}


We will construct a family of $(n+1)$-dimensional closed manifolds by defining embedded hypersurfaces in $\mathbb{R}^{n+2}$ and using the metric induced by the Euclidean metric. As we will see, they look like hook-shaped $(n+1)$-spheres, whose one branch is slightly swollen. The properties of these hooks will be useful to prove the two appearance theorems stated in the next subsection.

Consider a curve $\mu :[0,1]\to \mathbb{R}^2$ such that 
$$\mu(s) = \left\{ \begin{array}{rcl}
(1,s) & \mbox{for} 
& s\in [0,1/6],\\
(\cos(s\pi),1/6+\sin(s\pi)) & \mbox{for}
& s \in [1/3,2/3],\\
(-1,1-s) & \mbox{for} 
& s \in [5/6,1],
\end{array}\right.$$
and $\mu$ is chosen on $[1/6,1/3]\cup[2/3,5/6]$ so that it is a smooth curve.
For all integer $L>0$, consider the smooth curve $\gamma_L:[0,4]\to \mathbb{R}^2$ defined by
$$\gamma_L(s)=\left\{ \begin{array}{rcl}
(1,(s-1)L-1)& \mbox{for}
& s\in [0,1),\\
(1,s-2)& \mbox{for}
& s\in [1,2),\\
 \mu(s-2) & \mbox{for} & s\in [2,3],\\
(-1,(3-s)L) & \mbox{for} & s\in [3,4].
\end{array}\right. $$ 
It will be convenient to introduce the following function $f:[-1/2,1/2]\to \mathbb{R}$:
$$f(x)=\exp(1+\frac{1}{4x^2-1}) \quad \forall x\in[-1/2,1/2].$$ 
Let $d_0>0$ be smaller than half the focal radius of the curve $\mu$. Now, we identify $\mathbb{R}^2$ with $\mathbb{R}^2\times\{0\}$ in $\mathbb{R}^{n+2}$. At each point $p\in \gamma_L$, denote by $H[p]$ the normal hyperplane to $\gamma_L$ at $p$. For all $L>1$ and $\bar{\epsilon}=(\epsilon_1,\epsilon_2)\in(0,1/4)^2$, we choose a function $\phi[{L,\bar{\epsilon}}]: [0,4]\to \mathbb{R}^+$ such that:
\begin{itemize} 
\item $\phi[{L,\bar{\epsilon}}](s) = 1+\epsilon_1 f(s-1.5) \quad \text{for } s\in[1+\epsilon_2 , 2-\epsilon_2],$
\item $\phi[{L,\bar{\epsilon}}](1/L)=\phi[{L,\bar{\epsilon}}](4-1/L)=1$,
\item $\phi[{L,\bar{\epsilon}}]$ is increasing on $[0,1.5]$ and decreasing on $[1.5,4]$,
\item $\phi[{L,\bar{\epsilon}}]$ is strictly concave on $[0,1]$.
\end{itemize}
Define $$\Gamma[{L},{\bar{\epsilon}}] = \{x\in H[\gamma_L(s)]; s\in [0,4] , d(x,\gamma_L(s))=d_0\phi[{L,\bar{\epsilon}}](s)\}.$$ We can further impose that $\phi[{L,\bar{\epsilon}}]$ satisfies the following properties:
\begin{itemize}  
\item $\Gamma[{L,\bar{\epsilon}}]$ is a smooth closed hypersurface,
\item for all $L>1$ and $\epsilon_1\in(0,1/4)$ being fixed, $\Gamma[{L,\bar{\epsilon}}]$ converges smoothly to a hypersurface $\Gamma[{L,(\epsilon_1,0)}]$ and $\phi[{L,\bar{\epsilon}}]$ converges uniformly to a function $\phi[{L,(\epsilon_1,0)}]$ when $\epsilon_2\to 0$, 
\item the domains $\Delta[{L,\epsilon_1}] := \Gamma[{L,(\epsilon_1,0)}]\cap\{(x_1,...,x_{n+2}) ; x_2 < -L\} $ are all isometric to each other for $L>1$ and $\epsilon_1\in (0,1/4)$, and they have positive sectional curvature.  
\end{itemize}

In the following lemma, we list some useful properties of the manifold $\Gamma[{L,\bar{\epsilon}}]$ for any $n\geq1$:

\begin{lemme} \label{properties}
\begin{enumerate}
\item If $n>1$ and if $d_0$, $\epsilon_1$, $\epsilon_2$ are small enough, $\Gamma[{L,\bar{\epsilon}}]$ has (arbitrarily large) positive scalar curvature bounded below by a positive constant independent of $L$.

\item The manifold $\Gamma[{L,\bar{\epsilon}}]$ has positive sectional curvature on 
$$\Gamma[{L,\bar{\epsilon}}] \cap \{(x_1,...,x_{n+2}) ; x_1>0 \text{ and } x_2<-1\},$$
and on the open neighborhood of 
$$ \Gamma[{L,\bar{\epsilon}}] \cap \{ (x_1,...,x_{n+2}); x_1>0 \text{ and } x_2=-1/2\}$$
consisting of all points in $\Gamma[{L,\bar{\epsilon}}]$ at distance less than $\tilde{\delta}>0$ from the above set, where $\tilde{\delta}$ is independent of $L$ and $\bar{\epsilon}$.

\item Let $Z^a= \{x\in H[{\gamma_{L}(a)}] ; d(x,\gamma_{L}(a))=d_0\}$ and consider $Z^2$, $Z^{2.5}$ as hypersurfaces in $\Gamma[{L,(\epsilon_1,0)}]$. Then 
$$-\int_{Z^{2.5}} (R-\Ric(\nu,\nu)) > -\int_{Z^{2}} (R-\Ric(\nu,\nu)), $$
where $\nu$ denote a unit normal on these hypersurfaces and $R$ (resp. $\Ric$) is the scalar curvature of $\Gamma[{L,(\epsilon_1,0)}]$ (resp. its Ricci curvature) endowed with the metric induced by $\mathbb{R}^{n+2}$. 

\end{enumerate}
\end{lemme}

\begin{proof}
When $\epsilon_2\to 0$, $\Gamma[{L,\bar{\epsilon}}]$ converges to $\Gamma[{L,(\epsilon_1,0)}]$, and when $\epsilon_1\to 0$, $\Gamma[{L,(\epsilon_1,0)}]$ converges to a manifold called $\Gamma_L$. Hence to prove point $(1)$, it is enough to show that for $d_0$ small enough, $\Gamma_L$ has arbitrarily large positive scalar curvature. 
Since the scalar curvature on $\Gamma_L \cap \{(x_1,...,x_{n+2}) ; x_2<0\}$ is positive and arbitrarily large as $d_0$ goes to $0$, we only have to study
$$\Gamma_L \cap \{(x_1,...,x_{n+2}) ; x_2\geq 0\}.$$ 
But the desired property is clear since when $d_0$ goes to zero, the above subset converges after rescaling to a subset of a neck $S^n\times \mathbb{R}$ endowed with the product of a round metric and the standard metric on $\mathbb{R}$.


Point $(2)$ follows readily from the concavity of the function $\phi[{L,\bar{\epsilon}}]$ at the corresponding values.

The last point can be checked by computing the curvature for warped products, see \cite[Chapter 7, Corollary 43]{O'Neill} for instance (which holds for one-dimensional fibers). Indeed locally around $Z^2$ and $Z^{2.5}$, the metric is a warped product metric with base a round $n$-sphere of sectional curvature $d_0^{-2}$ and with fiber $[0,1]$. Let $f_w>0$ be the warping function for $Z^{2.5}$, the warping function for $Z^2$ being constant. On one hand
$$\int_{Z^2}(R-\Ric(\nu,\nu)) =\int_{Z^2} \frac{n(n-1)}{d_0^2},$$
on the other hand
\begin{align*}
\int_{Z^{2.5}}(R-\Ric(\nu,\nu)) & = \int_{Z^{2.5}}( \frac{n(n-1)}{d_0^2} -\frac{\Delta f_w}{f_w}) \\ & = \int_{Z^{2}} \frac{n(n-1)}{d_0^2} -\int_{Z^{2.5}}\frac{|\nabla f_w|^2}{f_w^2} \\ &< \int_{Z^2}(R-\Ric(\nu,\nu)).
\end{align*}

\end{proof}

\subsection{Appearance of stable geodesics and stable spheres}

In this subsection, we will use "stable sphere" (resp. "stable geodesic") to denote a closed stable embedded minimal $2$-sphere (resp. a simple closed stable geodesic).

\begin{theo} \label{appearance1}

There exists a two-sphere $(M,g)$ such that
\begin{enumerate}
\item $(M,g)$ does not contain stable geodesics, 
\item a stable geodesic appears along the Ricci flow starting at $(M,g)$. 
\end{enumerate}

\end{theo}

\begin{theo} \label{appearance2}
Let $M$ be a spherical space form $S^3/\Gamma$ which is endowed with a metric $g$ of positive scalar curvature. Suppose that $(M,{g})$ does not contain any stable sphere or embedded minimal $\mathbb{R}P^2$ with stable oriented double cover. Then for all point $p\in M$ and radius $r>0$, there is a metric $\tilde{g}$ on $M$ coinciding with $g$ outside $B_g(p,r)$ such that
\begin{enumerate}
\item $\tilde{g}$ has positive scalar curvature,
\item $(M,\tilde{g})$ does not contain any stable sphere or embedded minimal $\mathbb{R}P^2$ with stable oriented double cover, 
\item a stable sphere appears along the Ricci flow starting at $(M,\tilde{g})$. 
\end{enumerate}

\end{theo}

In the case where $M$ is two-dimensional, we clearly cannot assume its Gauss curvature to be positive at time $0$ since this property will be preserved along the Ricci flow and this will prevent the existence of stable geodesics.

From now on, we assume $d_0$, $\epsilon_1$, $\epsilon_2$ small enough so that by Lemma \ref{properties} (1), $R>0$ on $\Gamma[{L,\bar{\epsilon}}]$ when $n=2$. To prove Theorem \ref{appearance1}, we will need the following lemma. 

\begin{lemme} \label{nonex}
Let $n=1$. There exists a positive constant $C_0$ such that if $L>C_0$ and $\epsilon_1<1/C_0$, then for all $\epsilon_2$ sufficiently small the surface $\Gamma[{L,\bar{\epsilon}}]$ contains no stable geodesics.
\end{lemme}

\begin{proof}
Suppose by contradiction that there are two families $\{L_k\}_{k\in\mathbb{N}}$, $\{\bar{\epsilon}_{k,l} = (\epsilon_{1,k},\epsilon_{2,k,l})\}_{(k,l)\in \mathbb{N}^2}$ such that
$$L_k \to \infty \text{ as } k\to\infty$$
$$\epsilon_{1,k}\to 0\text{ as } k\to\infty$$
$$\forall k, \quad \epsilon_{2,k,l}\to 0 \text{ as } l\to\infty,$$
and a simple closed stable geodesic $S_{k,l}$ in $\Gamma[{L_k,\bar{\epsilon}_{k,l}}]$ for all ${(k,l)\in \mathbb{N}^2}$. We orientate a curve in $\Gamma[{L_k,\bar{\epsilon}_{k,l}}]$ of the form 
$$Z_{k,l}^s = \{
x\in H[{\gamma_{L_k}(s)}] ; d(x,\gamma_{L_k}(s))=d_0\phi[{L_k,\bar{\epsilon}_{k,l}}](s)
\}$$ where $s\in (0,4)$, by imposing that the outward normal $\nu$ is such that $\langle \nu,\gamma_{L_k}'(s)\rangle >0$. By construction, $\{Z_{k,l}^s\}_{s\in(1.5,4)}$ (resp. $\{Z_{k,l}^s\}_{s\in(0,1.5)}$) is a foliation of 

$$A^+_{k,l}=\Gamma[{L_k,\bar{\epsilon}_{k,l}}]\cap \{(x_1,x_2,x_{3}) ; x_1<0 \text{ or } x_2>-1/2\}$$
$$\text{(resp. } A^-_{k,l}=\Gamma[{L_k,\bar{\epsilon}_{k,l}}]\cap \{(x_1,x_2,x_{3}) ; x_1>0 \text{ and } x_2<-1/2\})$$
by concave (resp. convex) curves. Hence by the maximum principle, $S_{k,l}$ cannot be entirely contained in $A^+_{k,l}$ or in $A^-_{k,l}$. In other words, $S_{k,l}$ must intersect the central curve $Z_{k,l}^{1.5}$.

For a point $p\in \mathbb{R}^2$, we denote by $x_2(p)$ its second coordinate. Let $p_{k,l}$ be a point of $S_{k,l}$ such that $x_2(p_{k,l}) = \min_{p\in S_{k,l}} x_2(p)$. We already know by the previous paragraph that $x_2(p_{k,l})\leq -1/2$. By extracting a subsequence in $k$ and then in $l$ for each $n$, one can distinguish two situations:
\begin{enumerate}
\item there is a constant $\kappa_0>0$ independent of $k$, $l$ such that $x_2(p_{k,l})<-L_k-\kappa_0$
\item or $\liminf_{k\to \infty} [\inf_{l}(x_2(p_{k,l})+L_k)] \geq 0$.
\end{enumerate}
Recall the notation
$$\Delta[L_k,\epsilon_{1,k}]=\Gamma[{L_k,(\epsilon_{1,k},0)}]\cap \{(x_1,x_2,x_{3}) ; x_2<-L_k\}.$$
Suppose by contradiction that situation $(1)$ is true. Using the limit surfaces $\Gamma[{L_k,(\epsilon_{1,k},0)}]$ and the fact that the $\Delta[L_k,\epsilon_{1,k}]$ have positive Gauss curvature $K$ and are isometric, we infer that there is a constant $\kappa_2>0$ (independent of $k$) such that for all $k$:
$$\limsup_{l\to \infty} \int_{\tilde{S}_{k,l}} K > \kappa_2,$$
where $\tilde{S}_{k,l} = S_{k,l}\cap \Delta[L_k,\epsilon_{1,k}]$. Since $\Gamma[{L_k,\bar{\epsilon}}_{k,l}] \cap \{(x_1,x_2,x_{3}) ; x_1>0 \text{ and } x_2<-1\}$ has positive Gauss curvature and since we can choose $k$ so that the length of $S_{k,l}\cap  \{(x_1,x_2,x_{3}) ; x_1>0 \text{ and } x_2\in (-L_k,-1)\}$ is arbitrarily large, we can find a function $\phi$ on $S_{k,l}$ having a support included in $S_{k,l}\cap  \{(x_1,x_2,x_{3}) ; x_1>0 \text{ and } x_2 < -1\}$ such that
$$\int_{S_{k,l}} (|\nabla \phi|^2 - K \phi^2)<0$$  
for an $k$ sufficiently large and $l$ large in comparison. This contradicts the stability of the geodesic $S_{k,l}$.

We have to rule out situation $(2)$ by using the embeddedness of $S_{k,l}$. Let us show that the length of $S_{k,l}$ is necessarily bounded, for example by $6\pi d_0$, for $k$ large and $l$ large in comparison. Consider the subset $I_{k,l}$ of $S_{k,l}$ consisting of all the points in $S_{k,l}$ at distance less than $3\pi d_0$ to $p_{k,l}$, where the intrinsic distance of $S_{k,l}$ is used. Let $s_{k,l}$ be such that $x_2(\gamma_{L_k}(s_{k,l})) = x_2(p_{k,l})$. Since the tangent vector of $S_{k,l}$ at $p_{k,l}$ is orthogonal to $(0,1,0)$, and because of the geometry of the limit surfaces $\Gamma[{L_k,(\epsilon_{1,k},0)}]$, $I_{k,l}$ is an embedded multivalued graph with small gradient in $\Gamma[{L_k,\bar{\epsilon}_{k,l}}]$ over $Z_{k,l}^{s_{k,l}}$ for $k$, $l$ large. But the latter is close to a standard circle of radius $d_0$ for $k$, $l$ large so this situation is possible only if $I_{k,l}$ actually contains the whole geodesic $S_{k,l}$ and is a one-valued graph. Now that we have bounded the length of $S_{k,l}$ independently of $l$ for each $k$ large, and since each $S_{k,l}$ intersects $Z_{k,l}^{1.5}$, we can extract a subsequence in $l$ converging with multiplicity one to a stable geodesic $S_k$ in $\Gamma[{L_k,(\epsilon_{1,k},0)}]$ of length less than $6\pi d_0$. The sequence $\{S_k\}$ in turn converges subsequently in $\mathbb{R}^3$ to 
$$Z^{1.5}= \{x\in H[{\gamma_{L}(1.5)}] ; d(x,\gamma_{L}(1.5))=d_0\},$$
because $\Gamma[{L_k,(\epsilon_{1,k},0)}] \cap \{(x_1,x_2,x_{3}) ; x_1>0 \text{ and } -1<x_2<0\}$ becomes cylindrical as $k\to \infty$.
This is a contradiction with the stability assumption since in a neighborhood of $Z_{k,l}^{1.5}$ independent of $(k,l)$ (see Lemma \ref{properties} (2)), the sectional curvature of $\Gamma[{L_k,\bar{\epsilon}_{k,l}}]$ is positive.

\end{proof}

The next lemma is true for $1\leq n \leq 6$. We fix $\epsilon_1\in(0,1/4)$ and $L>1$. Let $\delta>0$ and define 
 $$Y^{\epsilon_2} =\Gamma[{L,(\epsilon_1,\epsilon_2)}] \cap \{(x_1,...,x_{n+2}) ; x_1<0 \text{ or } x_2>-\delta\}.$$
We choose $\delta\in (0,1/2)$ so that the boundaries $\partial Y^{\epsilon_2}$ are isometric and convex for all $0<\epsilon_2< \delta$. Define also
$$X^{\epsilon_2} = \Gamma[{L,(\epsilon_1,\epsilon_2)}]\cap \{(x_1,...,x_{n+2}) ; x_1<0 \text{ or } x_2>0\}
.$$
Similarly, we define $Y$ and $X$ by replacing $\Gamma[{L,(\epsilon_1,\epsilon_2)}]$ by $\Gamma[{L,(\epsilon_1,0)}]$ in the above formulas and we write $Y^0:=Y$, $X^0:=X$. Now for $\epsilon_2\in[0,\delta)$, suppose that $Y^{\epsilon_2}$ is isometrically embedded in a closed $(n+1)$-manifold $N^{\epsilon_2}$, in such a way that $N^{\epsilon_2}$ converges to $N^0$ as $\epsilon_2\to 0$. Let $\{N^{\epsilon_2}_t\}_{t\in[0,T)}$ be a solution of the Ricci flow starting at $N^{\epsilon_2}$ defined on a time interval $[0,T)$. If $V$ is a subset of $N^{\epsilon_2}$, let $V_t$ denote the Ricci flow at time $t$ starting at $V$ obtained by restriction of the original Ricci flow solution on $N^{\epsilon_2}$. By abuse of notations, we view $Y^{\epsilon_2}$ as a subset of $N^{\epsilon_2}$ in Lemma \ref{appears}.

In the proof, we will consider currents and varifolds in the closure $\bar{Y}$ which is isometrically embedded in $\mathbb{R}^{n+2}$. If $U$ is an open subset of $\bar{Y}$, the corresponding $(n+1)$-dimensional current will be called $[|U|]$ and if $C$ is an integral current, $|C|$ will be the name of the integer rectifiable varifold it determines by forgetting its orientation. If $k\in[|0,n+1|]$, the Grassmannian of $k$-planes in $\mathbb{R}^{n+2}$ is denoted by $\mathbf{Gr}(k,n+2)$ and its restriction to ${Y}$ is denoted by $\mathbf{Gr}(k,n+2,Y)$. The Hausdorff measure $\mathcal{H}^k$ of a subset of $Y_t$ is computed using the metric on $Y_t$.

\begin{figure} 
\includegraphics[scale=0.4]{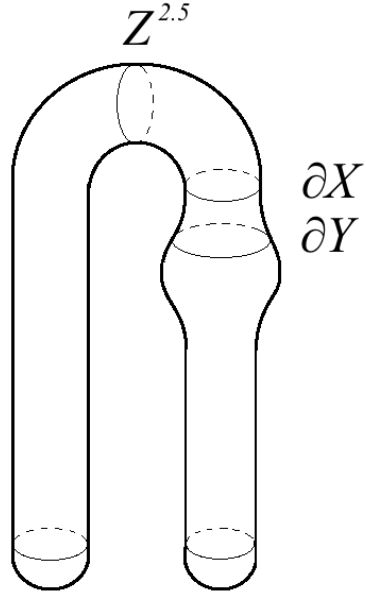}
\caption{When $\epsilon_1$ is small and $L$ large, $\Gamma[{L,(\epsilon_1,0)}]$ looks like a thin hook.}
\label{figA}
\end{figure}

\begin{lemme} \label{appears}
Suppose $1\leq n \leq 6$ and let $\epsilon_1\in(0,1/4)$. There exists a positive constant $C_1$ such that if $L>C_1$, then the following holds. For all $\epsilon_2 \in [0,\delta)$ small enough, there is a positive time $t_0=t_0(\epsilon_2)$ such that $Y^{\epsilon_2}_{t_0}$ contains an embedded stable minimal hypersurface. Moreover, $t_0$ can be chosen so that 
$$t_0(\epsilon_2)\to 0  \text{  as  } \epsilon_2\to 0.$$

\end{lemme}

\begin{proof}
From the proof, it will be clear that $\lim_{\epsilon_2\to 0} t_0(\epsilon_2) = 0$.

The boundary $\partial Y^{\epsilon_2}$ is convex with respect to the outward normal. Let $\Lambda^{\epsilon_2}$ be a sweepout in $Y^{\epsilon_2}_t$ associated with $X^{\epsilon_2}_t$ (see Definition \ref{definition}). We will show that when $\epsilon_2$ is small enough, for a positive time $t_0$ such that $\partial Y^{\epsilon_2}_{t_0}$ is still convex, we have:
\begin{equation} \label{width}
W(Y^{\epsilon_2}_{t_0},\partial Y^{\epsilon_2}_{t_0} , \Lambda^{\epsilon_2}_{t_0}) > \mathcal{H}^n(\partial X^{\epsilon_2}_{t_0}).
\end{equation}
Applying Theorem \ref{two boundaries} and Remark \ref{dimension 1}, we get an embedded minimal hypersurface $S^{\epsilon_2}$ in $Y^{\epsilon_2}_{t_0}$. If it is stable then the lemma is verified. If $S^{\epsilon_2}$ is not stable, then by minimizing its area in the connected open subset of $Y^{\epsilon_2}_{t_0} \backslash S^{\epsilon_2}$ whose boundary contains $ \partial Y^{\epsilon_2}_{t_0}$, we get an embedded stable hypersurface and the lemma is also verified in that case.

Hence to complete the proof, it remains to show (\ref{width}). Actually since the Ricci flow depends smoothly on the initial data, it is enough to check that if $\Lambda$ is the sweepout in $Y$ associated with $X$, then for all small positive times $t_0$, $\partial Y_{t_0}$ is convex and
$$W(Y_{t_0},\partial Y_{t_0} , \Lambda_{t_0}) > \mathcal{H}^n(\partial X_{t_0}).$$
For small times $\tau$, $\partial Y_{\tau}$ remains convex and subsequently we will only consider such small times. Assume by contradiction that for all small $\tau>0$, $W(Y_{\tau},\partial Y_{\tau} , \Lambda_{\tau}) = \mathcal{H}^n(\partial X_{\tau})$ and for each small $\tau>0$ let us choose a continuous sweepout $\{\Sigma^\tau_s\}_{s\in[0,1]}$ such that 
\begin{equation} \label{contrainte}
\max_s \mathcal{H}^n(\Sigma^\tau_s) \leq \mathcal{H}^n(\partial X_\tau) + \epsilon(\tau),
\end{equation}
where $\epsilon(\tau)$ is an arbitrary positive function converging to $0$ as $\tau$ goes to $0$ to be determined later. Let $\{\Omega^\tau_s\}_s$ be the family of open subsets of $Y_\tau$ associated with $\{\Sigma^\tau_s\}_s$ by Definition \ref{definition}. For $a \in [2,4-1/L]$, denote by $U^a$ the subset of $Y$ whose boundary (in $N^0$) is $Z^a $, where $Z^a= \{x\in H[{\gamma_{L}(a)}] ; d(x,\gamma_{L}(a))=d_0\}$. Let $\{\tau_k\}$, $\{s_k\}$ be two sequences such that $\tau_k\to 0$ and 
$$\mathcal{H}^{n+1}(\Omega^{\tau_k}_{s_k}) = \mathcal{H}^{n+1}(U^{2.5}).$$
We denote by $V^k$ the subset of $Y$ such that $V^k_{\tau_k}=\Omega^{\tau_k}_{s_k}$ (i.e. the open set of $N^0$ which becomes $\Omega^{\tau_k}_{s_k}$ at time $\tau_k$).
By \cite{FF}, we can choose $\{\tau_k\}$, $\{s_k\}$ so that $\partial [|V^{k}|]$ converges to an integral current ${C} = \partial [|V^\infty|]$ in the flat topology of $\bar{Y}$, where $\mathcal{H}^{n+1}(V^\infty) = \mathcal{H}^{n+1}(U^{2.5})$. We observe that (\ref{contrainte}) implies 
\begin{equation} \label{treza}
\mathbf{M}(C)\leq \omega_n d_0^n=\mathbf{M}(\partial [|U^{2.5}|])
\end{equation}
where $\omega_n$ is the volume of the $n$-dimensional unit round sphere. 

\textbf{Claim.} If $L$ was chosen large enough, then for $b\in[2,3]$, $\partial [|U^b|]$ is area minimizing among the currents $C'=\partial [|U'|]$ such that $$\mathcal{H}^{n+1}(U')= \mathcal{H}^{n+1}(U^{b}),$$
where $U'$ is an open subset relatively compact in $Y$, with a rectifiable boundary.

Let us prove this claim. Denote by $\hat{a}$ the function defined on $Y\backslash U^{4-1/L}$ such that $\hat{a}(x) = a$ if $\gamma_L(a)$ is the nearest point of $\gamma_L$ to $x$. Define the projection $\hat{p} : Y\backslash U^{4-1/L} \to Z^{4-1/L}$ such that $\hat{p}^{-1}(y)$ is exactly the line in $Y$ orthogonal to every $Z^a$ and beginning at $y\in   Z^{4-1/L}$. This projection enjoys the useful property of being area decreasing in the sense that if $R$ is a connected rectifiable set then 
$$\mathcal{H}^n(\hat{p}(R)) \leq \mathcal{H}^n(R)$$
with equality if and only if $R$ is included in a certain $Z^a$. Let $b$, $C'$, $U'$ be as above. To prove the claim, first note that when the support of $C'$ is contained in $Y\backslash U^{4-1/L}$ then we have $ \mathbf{M}(C') \geq \mathbf{M}(\partial [|U^{b}|])$ with equality if and only if $C'= \partial[|U^{b}|]$. Indeed, we can project $C'$ on $Z^{4-1/L}$ and get the current $\hat{p}_\sharp(C')$. By the constancy theorem, it is an integer multiple of $\partial [|U^{4-1/L}|]$. If it is non zero then $ \mathbf{M}(C') \geq \mathbf{M}(\partial [|U^{b}|])$. If it is zero then the varifold $\hat{p}_\sharp (|C'|)$ has mass at least twice $\mathcal{H}^n(\hat{p}(\partial U'))$ which has to be larger than $\frac{1}{2}\mathcal{H}^n(Z^{4-1/L})$ for large $L$: this is because if $A\subset Z^{4-1/L}$ has $n$-volume at most $\frac{1}{2}\mathcal{H}^n(Z^{4-1/L})$, then $\hat{p}^{-1}(A)$ has $(n+1)$-volume strictly less than $\mathcal{H}^{n+1}(U^{b})$ (for $L$ large). When the support of $C'$ is not contained in $U^{4-1/L}$, then by the coarea formula, there is a constant $\kappa$ independent of $L$ and an $a\in[3,3.5]$ which depends on $L$ such that $\spt(C')\cap Z^a$ is rectifiable and
$$\mathbf{M}(\langle C', \hat{a}, a\rangle ) \leq \kappa/L,$$
$$\mathbf{M}(\langle [|Y\backslash U'|], \hat{a}, a\rangle )  \leq \kappa/L,$$
where the notation for slicing is the same as in \cite[Chapter 2, \S\ 28]{Simon}. Consider the current $\hat{C}=\partial [|U^a \cup U'|]$. In fact, for $L$ large enough, 
\begin{equation} \label{stricte}
\mathbf{M}(\hat{C}) < \mathbf{M}(C').
\end{equation}
Indeed, by the monotonicity formula for minimal submanifolds, if $L$ is large then any area minimizing hypersurface in $Y$ with boundary $\langle C, \hat{a}, a\rangle$ must be contained in $Y\backslash U^{4-1/L}$ and so is equal to $\langle Y \backslash U', \hat{a}, a\rangle$ by the constancy theorem. Since $\spt \hat{C}\subset Y \backslash U^{4-1/L}$ and $\mathcal{H}^{n+1}(U^a \cup U') \geq \frac{1}{2}\mathcal{H}^{n+1}(Y) $ for large $L$, the previous argument shows that $ \mathbf{M}(\hat{C}) \geq \mathbf{M}(\partial [|U^{b}|])$. But then $C'$ has a bigger mass than $\partial[|U^{b}|]$ by (\ref{stricte}) as wished, and the claim is verified.

Consequently for $L$ large enough, (\ref{treza}) implies that the limit $C$ is actually $\partial [|U^{2.5}|]$ and that as $k\to \infty$, 
$$\mathbf{M}(\partial [|V^k|]) \to \mathbf{M}(\partial [|U^{2.5}|]).$$
By \cite[Chapter 2, 2.1, (18), (f)]{P}, the sequence of varifolds $|\partial [|V^{k}|]|$ converges subsequently to $|\partial [|U^{2.5}|]|$. 
Applying the definition of varifolds convergence to the function which sends $(x,H)\in\mathbf{Gr}(n,n+2,Y)$ to $-R+\Ric(\nu,\nu)$ where $\nu$ is a unit vector orthogonal to $H$ in $T_x\Gamma$, we have
$$\lim_{k\to \infty} \int_{\partial V^k} (-R+\Ric(\nu,\nu)) = \int_{Z^{2.5}} (-R+\Ric(\nu,\nu)),$$
which exactly means 
\begin{equation} \label{limit derivative}
\lim_{k\to \infty} \frac{\partial}{\partial t}\bigg|_{t=0} \mathcal{H}^n(\partial V^k_t) =
\frac{\partial}{\partial t}\bigg|_{t=0} \mathcal{H}^n(Z^{2.5}_t).
\end{equation}
To contradict inequality (\ref{contrainte}), we write the following Taylor expansions near $t=0$:
$$\mathcal{H}^n(\partial V^k_t) = \mathcal{H}^n(\partial V^k) + t.\frac{\partial}{\partial t}\bigg|_{t=0} \mathcal{H}^n(\partial V^k_t) + t^2 . \varphi_k(t),$$
$$\mathcal{H}^n(\partial X_t) = \mathcal{H}^n(\partial X) + t.\frac{\partial}{\partial t}\bigg|_{t=0} \mathcal{H}^n(\partial X_t) + t^2 . \phi(t),$$
where $\varphi_k$, $\phi$ are functions bounded independently of $k$ near $t=0$. 
By Lemma \ref{properties} (3),
\begin{equation} \label{grigri}
\frac{\partial}{\partial t}\bigg|_{t=0} \mathcal{H}^n(Z^{2.5}_t)>\frac{\partial}{\partial t}\bigg|_{t=0} \mathcal{H}^n(\partial X_t).
\end{equation}
Besides, the previous claim implies
\begin{equation} \label{grove}
\mathcal{H}^n(\partial V^k) \geq \mathcal{H}^n(Z^{2.5})=\mathcal{H}^n(\partial X).
\end{equation}
Hence, recalling that 
$$\partial V^k_{\tau_k} = \partial\Omega^{\tau_k}_{s_k} = \Sigma^{\tau_k}_{s_k}, $$
we combine (\ref{limit derivative}), (\ref{grigri}), (\ref{grove}) and the Taylor expansions to conclude for $k$ large:
$$\mathcal{H}^n(\Sigma^{\tau_k}_{s_k}) - \mathcal{H}^n(\partial X_{\tau_k}) > \frac{\tau_k}{2} \bigg{(} \frac{\partial}{\partial t}\bigg|_{t=0} \mathcal{H}^n(Z^{2.5}_t)-\frac{\partial}{\partial t}\bigg|_{t=0} \mathcal{H}^n(\partial X_t)\bigg{)}.$$
This is indeed the desired contradiction since the function $\epsilon(.)$ in (\ref{contrainte}) could converge arbitrarily fast to $0$, and this ends the proof.

\end{proof}

\begin{figure} 
\includegraphics[scale=0.4]{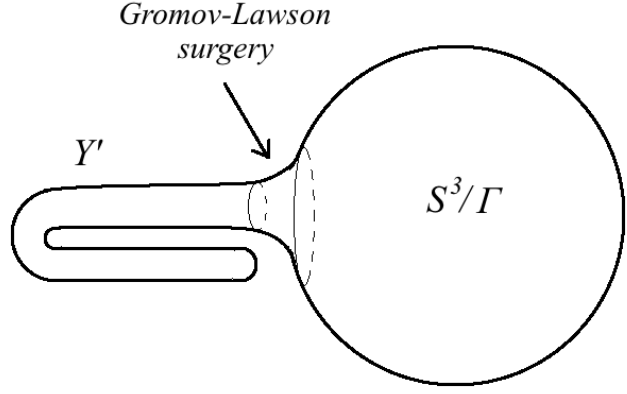}
\caption{Part of a thin hook, $Y'$, is glued to $S^3/\Gamma$ via the Gromov-Lawson procedure.}
\label{figB}
\end{figure}

\begin{proof}[Proof of Theorem \ref{appearance1} and Theorem \ref{appearance2}]

Theorem \ref{appearance1} follows from Lemmas \ref{nonex} and \ref{appears}, by taking $M=\Gamma[L,(\epsilon_1,{\epsilon_2})]=N^{\epsilon_2}$ with $\epsilon_1>0$, $\epsilon_2>0$ and $1/L$ sufficiently small.

To prove Theorem \ref{appearance2}, let $(M,g)$ be as in the statement. Choose any $p\in M$ and $r>0$ smaller than the injectivity radius of $(M,g)$ so that $\partial B_g(p,s)$ is convex whenever $0<s<r$. Let $r_0>0$ be smaller than $r$. Let $\epsilon_2$ be positive, consider a scaled-down version of $Y^{\epsilon_2}$ (as defined just before Lemma \ref{appears}) that we call $Y'$ and glue it to $M$ around $p$ by applying the Gromov-Lawson construction in $B_g(p,r_0)$ (see Figure \ref{figB}). We can modify the size of the following parameters: $r_0$, $\delta$, $\epsilon_2$, $Y'$. Analyzing how the forementioned construction is defined in \cite{GromovLawson} and taking the previous parameters small enough, we see that it can be done so as to get a new metric $\tilde{g}$ on $M$ verifying:
\begin{itemize}
\item $\tilde{g}$ coincides with $g$ outside $B_g(p,r_0)$,
\item $Y'$ is isometrically embedded in $(B_{{g}}(p,r_0),\tilde{g})$,
\item $(B_{{g}}(p,r),\tilde{g})$ is foliated by convex spheres,
\item the scalar curvature of $(M,\tilde{g})$ is bounded below by a positive constant independent of the parameters when the latter go to zero.
\end{itemize}
If the parameters are all small enough, then $(M,\tilde{g})$ contains no stable embedded minimal $2$-sphere or minimal $\mathbb{R}P^2$ with stable oriented double cover. Otherwise we could take the limit (subsequently by \cite{Sc}) and get a non-trivial oriented stable embedded minimal surface $S$ in $(M\backslash \{p\}, g)$ with finite area by the fourth item above (see \cite[Proposition A.1]{MaNe} for instance). By curvature estimates for stable surfaces (\cite{Sc}), the embedding is then proper, $S$ has finite Euler characteristic and the singularity at $p$ is removable (see \cite[Lemma 2.5]{LiZhou}, \cite[Proposition C.1]{ChodKetMax}, \cite{Gulliver} for instance). Hence the closure $\bar{S}$ is a smooth stable minimal $2$-sphere or a minimal $\mathbb{R}P^2$ with stable oriented double cover in $(M,g)$, contradicting our assumption. Finally, Lemma \ref{appears} ensures that a stable sphere appears along the Ricci flow starting at $(M,\tilde{g})$ provided $1/L$, $\epsilon_2$ are small enough.

\end{proof}

\subsection{Appearance of non-trivial singularities}


\begin{theo} \label{appearance3}
Let $M$ be a closed $3$-manifold satisfying the hypotheses of Theorem \ref{appearance2}. Then for all point $p\in M$ and radius $r>0$, there is a metric $\hat{g}$ on $M$ coinciding with $g$ outside $B_g(p,r)$ such that
\begin{enumerate}
\item $\hat{g}$ has positive scalar curvature,
\item $(M,\hat{g})$ does not contain any stable sphere or embedded minimal $\mathbb{R}P^2$ with stable oriented double cover, 
\item a non-trivial singularity occurs along the Ricci flow starting at $(M,\hat{g})$. 
\end{enumerate}
\end{theo}

For the appearance of stable spheres described in the previous subsection, a first order argument on the evolution of the thin hooks was enough. However we now want to study the long-time behavior and to prove that a non-trivial singularity occurs, we need to modify the metric of the hooks. We twist and stretch the bent part as follows. 
For all $\bar\epsilon$, $L$ and for all $a\in [2,3]$, using the previous notations we consider the subsets
$$Z^a[L,\bar\epsilon] = \{s\in H[\gamma_L(a)] ; d(x,\gamma_L(a)) = d_0\phi[L,\bar{\epsilon}](a)\} \quad a\in [2,3].$$
Let $\beta: [2,3]\to [0,1]$ be a bump function equal to zero in a neighborhood of $\{2,3\}$ and equal to one in [2+1/12,3-1/12]. Denote by $g_{eucl}$ the Euclidean metric in $\mathbb{R}^{n+2}$. For any $x\in Z^a[L,\bar\epsilon] $ (with $a\in [2,3]$), let $V(x)$ be a unit vector based at $x$ tangent to $\Gamma[L,\bar{\epsilon}]$ but normal to the hypersurface $Z^a[L,\bar\epsilon] $, in the metric induced by $g_{eucl}$. Then for any stretching factor $L_{st}\geq 0$, we define a new metric $g(L_{st})$ on $\Gamma[L,\bar{\epsilon}]$ such that it only differs from the metric induced by $g_{eucl}$ in $\bigcup_{a\in [2,3]} Z^a[L,\bar\epsilon] $ and for all $x\in \bigcup_{a\in [2,3]} Z^a[L,\bar\epsilon] $, $u,v\in T_x \Gamma[L,\bar{\epsilon}]\subset T_x\mathbb{R}^{n+2}$:
$$g(L_{st})(u,v) = g_{eucl}(u,v) + \beta(a) L_{st}  g_{eucl}(u,V(x))g_{eucl}(v,V(x)).$$

Hence the modified metric $g(L_{st})$ is similar to the metric induced by $g_{eucl}$, but strongly twisted and streched between $Z^{2}$ and $Z^{3}$ when $L_{st}$ is large. The choice of $\beta$ garantees that when $L_{st}$ goes to infinity and $\epsilon_2$ goes to zero, the other parameters being fixed, 
\begin{itemize}
\item $(\Gamma[L,\bar{\epsilon}],g(L_{st}))$ converges locally around $Z^{2+1/12}$ to $S^2\times \mathbb{R}$ endowed with the product metric $d_0^2 h_{1} + d\theta^2$, where $h_1$ is the round metric of Gauss curvature $1$,
\item $(\Gamma[L,\bar{\epsilon}],g(L_{st}))$ converges locally around $Z^{2.5}$ to a warped product metric $g_{tw}$ on $S^2\times \mathbb{R}$ different from a product metric.
\end{itemize}
Note nevertheless that in the second limit any slice $S^2\times\{\theta\}$ is also endowed with the round metric $d_0^2 h_1$. The Ricci flow for warped product metrics on $S^2\times \mathbb{R}$ which are $\mathbb{R}$-invariant (and hence with base $S^2$) has a well-controlled behavior and was studied in \cite{LottSesum}. For a metric $g$ on a $3$-manifold, let $T_{ext}(g)\in (0,\infty)$ be its extinction time when well-defined: when it exists it is defined as the time where a trivial singularity occurs. The following is a key lemma explaining why we consider these twisted hooks. 

\begin{lemme} \label{key}
Let $g_{inv}$ be an $\mathbb{R}$-invariant warped product metric on $S^2\times \mathbb{R}$. Suppose that $g_{inv}$ is not a product metric and that any slice $S^2\times\{\theta\} \subset S^2\times \mathbb{R}$ has area $4\pi d_0^2$ computed with $g_{inv}$. Then 
$$T_{ext}(g_{inv})> T_{ext}(d_0^2 h_1 +d\theta^2).$$
\end{lemme}

\begin{proof}
Note that this lemma is the long-time counterpart of Lemma \ref{properties} (3), which is a first-order property. The proof is essentially the same computation.

Let $\{g_{inv}(t) \}_{t\in [0,T_{ext}(g_{inv}))}$ the maximal solution starting at $g_{inv}$. We observe that the slices $S^2\times \{\theta\}$ remain totally geodesic for all times. Hence by the Gauss equation, their area $A(t)$ evolves according to
\begin{equation} \label{invariant}
\frac{d A}{d t} = - \int_{S^2\times \{\theta\}} (R-\Ric(\nu,\nu)) = -8\pi - \int_{S^2\times \{\theta\}} \Ric(\nu,\nu),
\end{equation}
where $\nu$ is a unit normal. Suppose now that $g_{inv}(t) = k(t) + e^{2u(t)} d\theta^2$, where $k(t)$, $u(t)$ are respectively a metric and a function on $S^2$. Then according to (2.4) in \cite{LottSesum}, the integral in the RHS of (\ref{invariant}) is equal to 
$$\int_{S^2} - |\nabla u(t)|^2 dvol_{k(t)}$$
where $\nabla$ and $|.|$ are computed using $k(t)$. Since $u(0)$ is not constant by hypothesis, $u(t)$ remains so and we obtain
$\frac{d A}{d t} > - 8\pi$.
Since the analogue derivative for a product metric is equal to $-8\pi$, and since the extinction time coincide with the time when the area of the slices $S^2\times \{\theta\}$ converges to $0$, we conclude that  
$$T_{ext}(g_{inv})> T_{ext}(d_0^2 h_1 +d\theta^2).$$
\end{proof}

Heuristically, to make a singularity appear, we will choose the stretching factor $L_{st}$ very large so that there are two regions evolving locally like two $\mathbb{R}$-invariant $S^2\times \mathbb{R}$, one of them being endowed with a product metric and separating the other one from a large region (to which we glued the twisted hook). Since the previous lemma suggests that the neck $S^2\times \mathbb{R}$ with a product metric should disappear first while the other regions stay large, a non-trivial singularity should occur. Let us make this reasoning rigorous with the following lemma.

\begin{lemme} \label{thin tube}
There exists a constant $\hat{C}>0$ and a time $\hat{T}>0$ such that the following holds. Let $(N,g(t))$, $0\leq t \leq t_1$ be a solution of the Ricci flow, assume that the initial metric $g(0)$ is normalized and that $N$ is a closed oriented connected $3$-manifold. Suppose that at $t_1$, $x\in N$ is in the center of a strong $\epsilon$-neck. Suppose that the center sphere of this strong $\epsilon$-neck separates $N$ into two components $N_1$, $N_2$ such that there are $x_i\in N_i$ ($i=1,2$) with
$$R(x,t_1) > \hat{C} (1+|R(x_i, t_1)|),$$
where $R(.,t_1)$ is the scalar curvature function at time $t_1$. 
Then, along the Ricci flow starting at $(N,g(0))$, a non-trivial singularity occurs before time $t_1 + \hat{T}$. 

\end{lemme}

\begin{proof}
We can suppose that $R(x,t_1)\geq r_0^{-2}$ where $r_0$ comes from the canonical neighborhood theorem (Theorem \ref{canoneighbo}). Consider the Ricci flow defined on a maximal time interval $ [0,T)$ where $t_1< T\leq \infty$; we want to show that a non-trivial singularity occurs at some $t_2$ larger than $t_1$. By definition of strong canonical neighborhoods, by (\ref{neck derivative}) and (\ref{derivative estimate}), since $\epsilon$ is small, there is a positive constant $C_1$ only depending on $\epsilon$ so that 
\begin{enumerate}
\item $\frac{\partial R(x,t)}{\partial t} \geq  R(x,t)^2/C_1$ as long as $x$ is in a strong $\epsilon$-neck,
\item either $R(x_i,t)\leq r_0^{-2}$ or $\frac{\partial R(x_i,t)}{\partial t} \leq C_2 R(x_i,t)^2$.
\end{enumerate}
The second item means that there is a time $t_3>t_1$ such that if the flow runs into a trivial singularity, then it does not occur before $t_3$. If the condition in the first item is verified as long as the classical Ricci flow is defined then $R(x,t)$ goes to infinity before a time $t_4$. Choose $\hat{C}$ large enough so that $t_4<t_3$. Suppose by contradiction that the point $x$ cease to be in a strong $\epsilon$-neck at time $t'\in[t_1,T)$ and the flow is well-defined on $[t_1,t')$. Then $(x,t')$ is in one of the following canonical neighborhoods:
\begin{itemize}
\item a $(C,\epsilon)$-cap (where in particular the scalar curvature is comparable at every point),
\item a $C$-component,
\item an $\epsilon$-round component.
\end{itemize}
Since $t'$ is the first time after $t_1$ such that $x$ is not in a strong $\epsilon$-neck, $(x,t')$ is actually in a $(C,\epsilon)$-cap. Either $x_1$ or $x_2$ is also in this cap. Now if $\hat{C}$ is large enough, then by (\ref{derivative estimate}) again each $R(x_i,t')$ cannot be comparable to $R(x,t')$ so this is a contradiction. Hence either the scalar curvature $R(x,t)$ goes to infinity before $t_4$ or a singularity happens elsewhere before $t_4$. Because $t_4<t_3$, this singularity is not trivial. By taking $\hat{T}=t_4$, it finishes the proof.

\end{proof}

\begin{proof}[Proof of Theorem \ref{appearance3}]
First, we glue a small twisted hook to $M$ around a point $p$ as in the proof of Theorem \ref{appearance2}. If $\epsilon_1$, $\epsilon_2$, $1/L$ and the size of the twisted hook are sufficiently small, then the new metric $\tilde{g}$ on $M$ does not contain any stable sphere or minimal  $\mathbb{R}P^2$ with stable oriented double cover and has positive scalar curvature. We take care of rescaling the new metric so that it becomes normalized. Let $(M_k,g_k)$ be a sequence of such rescalings, where the parameters $\epsilon_1$, $\epsilon_2$, $1/L$ and the size of the hook go to $0$. It is also possible to guarantee that for any sequence $\mathfrak{s}=\{x_k\}$ with $x_k\in M_k$, the based manifolds $(M_k,g_k,x_k)$ converge to one of the following geometric limits:

\begin{enumerate}[label=(\alph*)]
\item the flat $\mathbb{R}^3$ (corresponding to points $x_k$ not near the hook),
\item a rotationally symmetric non-compact $3$-manifold with two ends, one being a standard product metric on $S^2\times [0,\infty)$ with scalar curvature $1$ and the other one being a flat $\mathbb{R}^3\backslash B(0,1)$ (correspond to $x_k$ near the part where the hook is glued),
\item a warped product on $S^2\times \mathbb{R}$ with base a round $S^2$ with scalar curvature $1$ (corresponding to $x_k$ inside the hook far from the tip),
\item an $\mathbb{R}^3$ endowed with the standard initial metric (see \cite[Chapter 12]{MorganTian}) (corresponding to $x_k$ near the tip of the hook).
\end{enumerate}
Let $T$ be the maximum of the maximal times for which the Ricci flows starting at one of these four metrics are smoothly defined. By \cite{ChenZhu} and \cite{Shinoncompact}, the hypotheses of Lemma \ref{geomconv} are satisfied. Notice that if $x_k\in Z^{2+1/12}$ (resp. $Z^{2.5}$) for all $k$ then the geometric limit is a product metric (resp. non-trivial warped product metric) on $S^2\times \mathbb{R}$, whose life span under the Ricci flow is equal (resp. strictly longer) than that of the standard initial metric by \cite[Theorem 12.5]{MorganTian} (resp. Lemma \ref{key}).

Suppose by contradiction that no non-trivial singularity occurs along the Ricci flow starting at $(M,\tilde{g})$.  Two cases are a priori possible: $T=1$ the life span of the standard initial metric, or $T<1$. The latter situation corresponds to the maximum of the scalar curvature being reached around the gluing part near $T$, namely it means that $T$ is the maximal existence time for the second geometric limit in the previous list. Note that this Ricci flow being rotationally symmetric with two ends, the only canonical neighborhood that can appear is a strong $\epsilon$-neck. By the above remarks and Lemma \ref{geomconv}, in both cases one finds $\delta>0$ so that for all $k$ large, the Ricci flows $(M_k,g_k(t))$ have no singularity until at least $\hat{t} := T-\delta$, time at which for some $q,q_1\in M_k$, and for any $q_2\in Z^{2.5}$:
\begin{itemize}
\item for $i=1,2$, $R(q,\hat{t}) > \hat{C} (1+R(q_i,\hat{t}))$ ($\hat{C}$ being the constant in Lemma \ref{thin tube}),
\item $q$ is in a strong $\epsilon$-neck whose central sphere separates $q_1$, $q_2$.
\end{itemize}
Actually, $q_1$ is chosen to be a point of $M$ far from $p$ where the gluying is realized in the original metric $g$. The hypothesis of Lemma \ref{thin tube} are satisfied and a non-trivial singularity occurs, which contradicts our assumption that only a trivial singularity occurs.

\end{proof}

\begin{remarque}
\begin{enumerate}
\item In the proof of Theorem \ref{appearance2}, we used hooks with a strectching factor $L_{st}=0$ for simplicity. However, it is not difficult to check that $\bar{\epsilon}$ and $L$ can be chosen so that for any stretching factor $L_{st}$, Lemma \ref{appears} remains true. Hence putting Theorems \ref{appearance2} and \ref{appearance3} together, we conclude that there are $3$-manifolds with positive scalar curvature such that along the Ricci flow, a stable sphere appears and some time later, a non-trivial singularity occurs.
\item Although according to Theorem \ref{appearance3}, a non-trivial singularity occurs in certain cases, it does not provide information on where it happens: intuitively one expects the singularity to occur at the neck with a product metric or at the tip of the twisted hook or at both places, depending on the shape of the tip.
\end{enumerate}
\end{remarque}

\section{Stable spheres and Type I singularities} \label{type I}

In \cite{AngenentKnopf}, examples of rotationally symmetric $S^{n+1}$ developing a Type I neckpinching are constructed. Actually in dimension $3$, this is part of a much more general fact. By joining any two oriented $3$-manifolds with a thin neck, we obtain an initial data which will produce a non-trivial Type I singularity under the Ricci flow. 

\begin{prop}\label{singularity neck}
Let $(M_1,g_1)$, $(M_2,g_2)$ be two closed oriented $3$-manifolds. For any pair of points $p_i\in M_i$ ($i=1,2$), radius $\hat{r}>0$ small enough, length $l\geq0$ and $\delta>0$, there exists a metric $g$ on the connected sum $M=M_1\#M_2$ such that:
\begin{enumerate}
\item there is a subset $N\subset M$ diffeomorphic to $S^2\times (0,1)$ so that $M\backslash N$ is isometric to $(M_1\backslash B(p_1,\hat{r})) \cup(M_2\backslash B(p_2,\hat{r}))$,
\item $M$ is $\delta$-close in the Hausdorff-Gromov distance to the union of $M_1$, $M_2$ and a curve of length $l$ joining $p_1$ to $p_2$, 
\item a non-trivial singularity of Type I occurs along the Ricci flow starting at $(M,g)$.
\end{enumerate}
\end{prop}

\begin{proof}
As previously the proof is a limiting argument. We can glue an arbitrarily thin neck joining $M_1$ and $M_2$ so that $M$ is $\delta$-close in the Hausdorff-Gromov distance to the union of $M_1$, $M_2$ and a curve of length $l$ joining $p_1$ to $p_2$. We can ensure that this gluying is done locally around $p_i$, which does not affect the original metric in $(M_1\backslash B(p_1,\hat{r})) \cup(M_2\backslash B(p_2,\hat{r}))$. Let $h_k$ be a sequence of metrics corresponding to thinner and thinner such necks. Let $Q_k$ be the maximum of the scalar curvature on $(M,h_k)$, that we assume is achieved at the middle of the neck. Denote by $\tilde{h}_k$ the rescaling $Q_k h_k$, and let $\tilde{h}_k(t)$, $0\leq t\leq T_k$, be a maximal solution for the Ricci flow. We choose the sequence of metrics so that for any sequence of points $x_k\in M$, the rescalings $(M,\tilde{h}_k(0),x_k)$ subsequently converge geometrically to either a flat $\mathbb{R}^3$, or a product metric on $S^2\times \mathbb{R}$ or a limit of type $(b)$ described in the proof of Theorem \ref{appearance3}.  Then by Lemma \ref{geomconv}, for $k$ large, there is a point $x$ which was in the neck at time $0$, is in a strong $\epsilon$-neck with arbitrarily large scalar curvature (in particular at least $r_0^{-2}$) at a certain time $t'$ independent of $k$. Notice that the rescalings at points in $(M_1\backslash B(p_1,\hat{r})) \cup(M_2\backslash B(p_2,\hat{r}))$ converge geometrically to a static flat $\mathbb{R}^3$. Hence by Lemma \ref{thin tube}, for every large $k$ a non-trivial singularity occurs at time $T_k\in(t',t'+\hat{T})$.

\textbf{Claim:} If $k$ is large enough a $(C,\epsilon)$-cap with scalar curvature at least $2r_0^{-2}$                  cannot appear during the Ricci flow $(M,\tilde{h}_k(t))$, $0\leq t\leq T_k$. 

Suppose the claim to be true, then the singularity is of Type I according to \cite{YDing}. The theorem is thus proved modulo the claim.

To verify the claim, let us consider a sequence $\{(M,\tilde{h}_{k(l)}(t))\}_l$ of counterexamples. For each $l$, let $t_l$ (resp. $s_l$) be the infimum of the times at which there is a $(C,\epsilon)$-cap with scalar curvature at least $2r_0^{-2}$ (resp. $r_0^{-2}$), for the metric $\tilde{h}_{k(l)}(t_l)$. We can suppose that $T_{k(l)}$ (resp. $t_{k(l)}$, $s_l$) converges to $T_\infty$ (resp. $t_\infty$, $s_\infty$). Actually we have $s_\infty < t_\infty\leq T_\infty$. Indeed note that the only kinds of canonical neighborhoods with large scalar curvature that can appear are $\epsilon$-necks which diffeomorphic $S^2\times(0,1)$ and $(C,\epsilon)$-caps which are diffeomorphic to a ball or $\mathbb{R}P^3$ minus a point. For this reason, there is a $(C,\epsilon)$-cap at time $t_l$ with scalar curvature at least $2 r_0^{-2}$ and by tracking this region we can go back in time to find a $(C,\epsilon)$-cap with scalar curvature at least $ r_0^{-2}$ at time $t_l-\delta$. In view of the derivative estimate (\ref{derivative estimate}) this delta can be chosen independent of $l$, and we get $s_\infty < t_\infty$ as desired. Next we pick $p_l$ a point in a $(C,\epsilon)$-caps at time $s_l$. By definition of $t_l$ and by (\ref{neck derivative}), the curvature tensor is uniformly bounded on $[0,(s_l+t_l)/2]$. Recall that the pointed zero time slices $(M,\tilde{h}_{k(l)}(0),p_l)$ converge to a limit $(M_\infty, \tilde{h}_\infty(0),p_\infty)$ with bounded curvature so the flow starting at this limit exists and is unique \cite{Shinoncompact} \cite{ChenZhu}. Let $S>0$ such that $(M_\infty, \tilde{h}_\infty(t),p_\infty)$ is maximally defined on $[0,S)$. By construction this limit flow is rotationally symmetric non-compact with two ends when non-flat, so the only canonical neighborhoods with large curvature which could appear are strong $\epsilon$-necks. Actually by Lemma \ref{geomconv}, $S\geq \frac{s_\infty+t_\infty}{2}$. Indeed otherwise for $l$ large and $t''$ close to $S$ there should be an arbitrarily thin $\epsilon$-neck for the metric $\tilde{h}_{k(l)}(t'')$ but then (\ref{neck derivative}) and (\ref{derivative estimate}) would contradict $\frac{s_{l}+t_{l}}{2}>S$. So $S\geq \frac{s_\infty+t_\infty}{2}$ and by Lemma \ref{geomconv} again, $(M_\infty, \tilde{h}_\infty(s_\infty))$ should then contain a $(C,\epsilon)$-cap, which is impossible and our claim is proved. 

\end{proof}

From the proof of Proposition \ref{singularity neck}, it can be shown for the examples where a Type I singularity appears at some time $t_1$ that for all $t\in[0,t_1)$ there is an embedded stable minimal sphere $S(t)$ whose area goes to $0$ as $t\to t_1$. These spheres correspond to the neckpinching. One can wonder if this is a general phenomenon. The next theorem confirms that indeed small stable spheres or $\mathbb{R}P^2$ with stable oriented double cover are closely related to Type I singularities. When a minimal surface is an embedded stable sphere or an embedded $\mathbb{R}P^2$ with stable oriented double cover, we will call it a \textit{stable immersed sphere with embedded image}.

\begin{theo} \label{spherestypeI}
Let $M$ be an oriented closed connected $3$-manifold. Consider a Ricci flow $(M,g(t))$, $0\leq t < T$ and suppose that there is a non-trivial Type I singularity at time $T$. Then for all time $t$ close to $T$, $(M,g(t))$ contains a stable immersed sphere with embedded image $S(t)$ such that 
$$C'(T-t) \leq \mathcal{H}^2(S(t)) \leq C''(T-t),$$
where $C'$, $C''$ are constants independent of $t$.

Conversely, suppose that there is a sequence of times $s_k$ converging to $T$ and a sequence of stable immersed spheres with embedded image $S_k$ in $(M,g(s_k))$. Suppose also that the area of $S_k$ goes to zero and
$$\mathcal{H}^2(S_k) \geq C'(T-s_k),$$
where $C'$ is a constant independent of $k$. Then there is a singularity at time $T$ and it is locally of Type I in the following sense: 
\begin{align*}
& \forall A>0, \exists \bar{C}=\bar{C}(C',A,C,\epsilon),\forall k >\bar{C}, \\
&  \sup\{|Rm(x,s_k)| ; (\max_{S_k}R).d(x,S_k)^2 \leq A\} \leq \frac{\bar{C}}{T-s_k}.
\end{align*}
\end{theo}

\begin{proof}
Without loss of generality we assume $(M,g(0))$ to be normalized. Suppose that $(M,g(t))$, $0\leq t < T$, develops a Type I singularity at $T$. Since the singularity is non-trivial, for all times close to $T$, say for $t\in (t_0,T)$, there is a constant $A>0$ independent of $t$ such that the points of scalar curvature larger than $A$ are in strong $\epsilon$-necks or in $(C,\epsilon)$-caps diffeomorphic to $\mathbb{R}P^3$ minus a point according to \cite{YDing}. By \cite[Proposition A.21]{MorganTian}, this means that at time $t\in[t_0,T)$ two situations can happen:
\begin{itemize}
\item $M$ is covered by the previous canonical neighborhoods and is diffeomorphic to $S^2\times S^1$ or $\mathbb{R}P^3\# \mathbb{R}P^3$. The existence of a stable immersed sphere with embedded image $S(t)$ is obtained by $\gamma$-reduction \cite{MSY}, furthermore $S(t)$ has area going to zero as $t$ goes to $T $,
\item $(M,g(t))$ contains $T(t)$, an $\epsilon$-tube or a $C$-capped $\epsilon$-tube (diffeomorphic to $\mathbb{R}P^3$ minus a point), whose curvature at the end(s) is at most $A$ but which contains points whose scalar curvature goes to infinity as $t$ approaches $T$. 
\end{itemize}
In the second case, choose a sphere $Z(t)$ in $T(t)$, which is the central sphere of a strong $\epsilon$-neck and with area going to zero if $t$ is close to $T$. We can try to minimize its area in $T(t)$ because the boundary component(s) of $T(t)$ have large area in comparison. Actually $Z(t)$ is homologically non-trivial in $T(t)$ and one cannot reduce its area to zero by isotopies. By deforming slightly the boundaries of $T(t)$ to make them strictly mean-convex, we can use $\gamma$-reduction again to find a stable immersed sphere with embedded image $S(t)$. For times close to $T$, this minimal surface is far from the boundaries where we deformed the metric by the monotonicity formula and the geometry of the necks, so it is in fact minimal for the original metric $g(t)$ and 
$$\mathcal{H}^2(S(t)) \leq \mathcal{H}^2(Z(t)) \to 0  \text{  as  } t \to T.$$
Choose $S(t)$ to be of least area among stable immersed spheres with embedded image at time $t$ close to $T$. Notice that the scalar curvature on $S(t)$ is comparable everywhere to the maximum of the scalar curvature on $(M,g(t))$ by the choice of $S(t)$ and the canonical neighborhood theorem. But it is known that for a Type I singularity the scalar curvature blows up in $\frac{1}{T-t}$ hence the area of $S(t)$ decreases to zero linearly and the first part of the theorem is proved.

For the second part, we can argue as follows. Let $p_k\in S_k$ a point where the scalar curvature achieves its minimum on $S_k$, then by \cite[Theorem 11.19]{MorganTian} and the monotonicity formula, since the area of $S_k$ converges to $0$, $R(p_k,s_k)$ goes to infinity. By the area upper bound (depending on the scalar curvature) (\cite[Proposition A.1]{MaNe}) and curvature bound for stable spheres \cite{Sc}, by the classification of canonical neighborhoods and their properties, for $k$ large, $S_k$ has to be a sphere or $\mathbb{R}P^2$ entirely contained in a strong $\epsilon$-neck or in a $(C,\epsilon)$-cap diffeomorphic to $\mathbb{R}P^3$ minus a point. The area bound from below for $S_k$ implies that the scalar curvature on $S_k$ is smaller than $C''/(T-s_k)$ for a certain constant $C''$. The conclusion now follows from the "bounded curvature at bounded distance" property (\cite[Chapter 10]{MorganTian}).

\end{proof}

\begin{remarque} \label{minimum de R}
From the proof of the previous theorem, it becomes clear that when there is a sequence of stable immersed spheres with embedded image $S_k$ at times $s_k$ going to $T$, with area converging to $0$, then the minimum of the scalar curvature on these spheres, $\min_{S_k} R$, goes to infinity and $\max_{S_k}R/\min_{S_k} R$ is bounded.
\end{remarque}

\section{Symmetry and non-appearance of stable spheres}
In this section, we study under which symmetry assumptions one can rule out the appearance of stable immersed spheres with embedded image along the Ricci flow. In the case of a finite group $G$ acting effectively by isometries on a $3$-manifold, there is a point $p$ which is fixed only by the identity and one can glue disjoint thin hooks at the images of $p$ under the elements of $G$, in an equivariant way. This gives a $G$-invariant metric for which stable spheres appear along the Ricci flow. Hence, we will only focus on positive dimensional compact Lie groups. Consider $(M,g)$, an oriented connected closed $3$-manifold on which a $d$-dimensional compact Lie group $G$ of isometries acts effectively. Assume that 
\begin{enumerate} 
\item either $d>1$, 
\item or $d=1$ and the action is free.
\end{enumerate}
We will say that $(M,g)$ (as above) is rotationally symmetric if a subgroup $G_0$ of $G$ is isomorphic to $SO(3)$ and there is a $G_0$ invariant $2$-sphere or $\mathbb{R}P^2$ embedded in $(M,g)$. This amounts to saying that a cover of $(M,g)$ is a warped product $I\times S^2$ with fiber $S^2$, where $I=\mathbb{R}$ or $I=[0,1]$ (the warped product is then degenerate at $0$, $1$). In that case, $M$ is diffeomorphic to $S^3$, $\mathbb{R}P^3$, $\mathbb{R}P^3 \# \mathbb{R}P^3$ or $S^2\times S^1$.

\begin{theo} \label{symmetry}
Let $(M,g)$ be as above. Suppose that it contains no stable immersed spheres with embedded image. Then, along the Ricci flow starting at $(M,g)$, stable immersed spheres with embedded image cannot appear.
\end{theo}

\begin{proof}
Let $(M,g(t))$, $0\leq t \leq T$, $g(0)=g$, be a solution of the Ricci flow and suppose by contradiction that there is a stable immersed sphere with embedded image $S(t_1)$ in $(M,g(t_1))$. By uniqueness of the Ricci flow, $G$ still acts by isometries on $(M,g(t_1))$. For all $V\in\mathfrak{g}$ a vector in the Lie algebra of $G$, we define $\phi_V(s)$ the $1$-parameter family of diffeomorphisms of $G$ generated by the left-invariant vector field corresponding to $V$. Note that for any $V\in\mathfrak{g}$, the projection of
$$\frac{d(\phi_V(s).x)}{ds} \in T_x M \quad  x\in S(t_1)$$
on the normal bundle of $S(t_1)$ is a Jacobi field $J_V$. By stability either it is identically zero or it does not vanishes. Let us show that $J_V$ has to be zero. In the case where $S(t_1)$ is an $\mathbb{R}P^2$ it is clear since its normal bundle is non-trivial ($M$ is oriented). If $S(t_1)$ is an embedded sphere and $J_V\neq 0$ , then in a neighborhood of $S(t_1)$, $\{\phi_V(s).(S(t_1))\}_{s\in[0,s_0]}$ foliates one side of $S(t_1)$ as long as $\phi_V(s_0).(S(t_1)$ does not touch $\phi_V(0).(S(t_1)$ from the other side. When it does so at $s_0$, by minimality, the two surfaces coincide: $\phi_V(0).(S(t_1) = \phi_V(s_0).(S(t_1)$. Since such an $s_0>0$ exists in the case where $J_v$ is not identically zero, we deduce by connectedness and orientability of $M$ that $M$ is an $S^2 \times S^1$, a contradiction since for topological reasons it always contains a stable sphere. 

We just proved that for all $x\in S(t_1)$, the vector $X=\frac{d(\phi_V(s).x)}{ds}$ is tangent to the sphere $S(t_1)$ for all $V\in\mathfrak{g}$. It means that $G$ acts on $S(t_1)$. Since any compact $1$-dimensional group of isometries acting on a $2$-sphere or $\mathbb{R}P^2$ fixes a point, $G$ is of dimension $d>1$ so Case (2) is proved. For Case (1), since $G$ is of dimension greater than $1$ and acts effectively by isometries on a $2$-sphere, the connected component $G_0$ containing ${\Id}$ is isomorphic to the rotation group $SO(3)$. In otherwords, $M$ is rotationally symmetric and then the non-appearance of stable spheres along the Ricci flow is reduced to an ODE argument. By \cite[Theorem A]{Angenentnodal}, stable spheres invariant under $G_0$ cannot appear if there were none at the beginning and we can check that any stable sphere, if it exists, is $G_0$-invariant. The assumption that a stable sphere appears is thus absurd. The situation for $\mathbb{R}P^2$ with stable oriented double cover is similar.

\end{proof}

\begin{remarque} \label{petiteremarque}
\begin{enumerate}
\item Note that the $3$-dimensional (twisted) hooks defined in Section \ref{appear1} have an effective $S^1$-action which is not free, so according to Theorem \ref{symmetry} these examples where stable spheres appear have in some sense a maximal amount of symmetry.  
\item A byproduct of the proof of Theorem \ref{symmetry} is that if $(M,g)$ (as above) contains a stable immersed sphere with embedded image $S$ and if $(M,g)$ is not rotationally symmetric, then it is an $S^2\times S^1$ foliated by stable spheres which are images of $S$ under a family of isometries.

\end{enumerate}
\end{remarque}

\begin{lemme} \label{gluck}
Let $(M,g)$ be as above. If a Type II singularity occurs then $(M,g)$ is a rotationally symmetric sphere or $\mathbb{R}P^3$.
\end{lemme}

\begin{proof}
Let $t_1$ be the time of a Type II singularity. By \cite{YDing}, just before $t_1$, there is a region of high scalar curvature which is a $(C,\epsilon)$-cap diffeomorphic to a $3$-ball. By \cite[Lemma 14.3.11, Proposition 14.3.12]{MorganTianII}, the action of $H$ is equivariant to a linear action and there is a fixed point. Consequently, $H$ cannot be $1$-dimensional by our assumption on $G$ and $M$ is a rotationally symmetric sphere or $\mathbb{R}P^3$. 
\end{proof}

Because of the link between Type I singularities and stable spheres described in Section \ref{type I}, we readily obtain the following corollary. 

\begin{coro}  \label{cocoo}
Let $(M,g)$ be as above. The following holds along the Ricci flow.
\begin{enumerate}
\item When $M$ is a rotationally symmetric $3$-sphere and does not contain stable spheres, then no non-trivial Type I singularity occurs. 
\item When $M$ is a rotationally symmetric $\mathbb{R}P^3$ and does not contain stable immersed spheres with embedded image, then no non-trivial Type I singularity occurs. 
\item When $M$ is rotationally symmetric and neither a $3$-sphere nor an $\mathbb{R}P^3$, no Type II singularity occurs.
\item When $M$ is not rotationally symmetric and if a singularity occurs, then it is a Type I trivial singularity. 
\end{enumerate}

\end{coro}
\begin{proof}
The first item comes from Theorem \ref{spherestypeI} and Theorem \ref{symmetry}, the second item is proved in the same way considering a double cover. Lemma \ref{gluck} yields the third item. For the fourth item, a singularity must be of Type I by Lemma \ref{gluck} and \cite{YDing}. Let $T$ be a time of singularity. Suppose that the singularity is non-trivial, then by Theorem \ref{spherestypeI} and Remark \ref{petiteremarque} (2), $(M,g(t))$ is an $S^2\times S^1$ foliated by small spheres for all $t$ close to $T$. The curvature blows up everywhere in that case (see Remark \ref{minimum de R}), contradicting our assumption and the corollary is verified.  
\end{proof}

A question still left unanswered is whether a Type II singularity can appear in the case of item $(1)$. In item $(3)$, the other kinds of singularities can occur. The first item was proved in \cite{GuZhu} for all dimensions. In the case of a free $S^1$ action, it was already suggested in \cite[Remark 2.6]{LottSesum} to combine the singularity analysis with the symmetry.

\bibliographystyle{plain}
\bibliography{biblio17_04_04}

\end{document}